\documentclass[article]{amsart}
\usepackage{amsfonts}
\usepackage{enumerate}
\begin{document}


\title[On Minimal Sums of Tensor Product of Sequences]{On Minimal Sums of Tensor Product of Sequences in Separable Hilbert Spaces}


\author{Abdelkrim Bourouihiya}
\address{Department of Mathematics, Nova Southeastern University}
\curraddr{3301 College Avenue, Fort Lauderdale, Florida, USA}
\email{ab1221@nova.edu}

\author{Samir Kabbaj}
\address{Department of Mathematics, University of Ibn Tofail}
\curraddr{Kenitra, Morocco}
\email{samkabbaj@yahoo.fr>}






\keywords{Hilbert spaces, Bessel Sequences, Frames, Tensor products}


\subjclass{ Primary 42C15; Secondary 47A80}{}



\newtheorem{thm}{Theorem}[section]
\newtheorem{cor}[thm]{Corollary}
\newtheorem{lem}[thm]{Lemma}
\newtheorem{prop}[thm]{Proposition}
\newtheorem{ax}{Axiom}

\theoremstyle{definition}
\newtheorem{defn}[thm]{Definition}
\newtheorem{examp}[thm]{Example}
\newtheorem{rem}[thm]{Remark}
\newtheorem{conj}[thm]{Conjecture}
\newtheorem*{notation}{Notation}

\newtheorem{exercise}[thm]{Exercise}
\newtheorem{conclusion}[thm]{Conclusion}
\newtheorem{criterion}[thm]{Criterion}
\newtheorem{summary}[thm]{Summary}
\newtheorem{problem}[thm]{Problem}
\numberwithin{equation}{section}

\begin{abstract}
It is known that the tensor product of two sequences, in the tensor product of two separable Hilbert spaces, is a frame  if and only if each component of that product is a frame. This paper proposes a sort of generalization of the aforementioned result by dealing with sequences S that are finite minimal sums of tensor products of a finite number of sequences. We prove that S is a Bessel sequence if and only if it is a sum for which each term is the tensor product of Bessel sequences. We also state necessary conditions for S to be a frame. For dimensions d greater than one, we deduce several  results on Gabor systems generated by finite rank square integrable functions.  Surprisingly, the one dimensional versions of some of these results are extremely difficult to prove or disapprove. 
\end{abstract}

\maketitle

\section{Introduction}
For this paper, all Hilbert spaces are assumed to be separable and denoted by $\mathcal{H}$ or $\mathcal{K}$.  The space of bounded operators: $\mathcal{H}\rightarrow \mathcal{K}$ is denoted $\mathcal{B}(\mathcal{H}, \mathcal{K})$.

\begin{defn}\label{frame}
Let $\mathcal{F} = \{f_n\}_{n >0}$ be a sequence in $\mathcal{H}$.
\begin{enumerate}[(i)]
  \item $\mathcal{F}$ is a \emph{Bessel sequence} in $\mathcal{H}$ if there exists a constant $B > 0$, called a Bessel bound, such that
  \begin{eqnarray*}
  \forall f \in \mathcal{H}, \quad  \sum_{n>0} |\langle f, f_n \rangle |^2 \leq B \|f \|^2.
  \end{eqnarray*}
  \item $\mathcal{F}$ is a \emph{frame} for $\mathcal{H}$ if there exist constants $A,B > 0$, called frame bounds, such that
  \begin{eqnarray*}
  \forall f \in \mathcal{H}, \quad  A \|f \|^2 \leq \sum_{n>0} |\langle f, f_n \rangle |^2 \leq B \|f \|^2.
  \end{eqnarray*}
  \item $\mathcal{F}$ is a \emph{Riesz basis} if  $\mathcal{F}$ is a frame that is the image of an Orthonormal Normal  Basis (ONB) under a continuous and  invertible linear map of  $\mathcal{H}$ onto itself.
\end{enumerate}
\end{defn}

A \emph{translation} of $g \in L^2(\mathbb{R}^{d})$ by $a \in \mathbb{R}^d$ is $T_a g(x) = g(x-a)$, a \emph{modulation} of $g$ by $b \in \mathbb{R}^d$ is $M_b g(x) = e^{2\pi i\langle x,b \rangle} g(x)$, and a \emph{dilation} of $g$ by $\lambda \in (0,\infty)^d$ is $D_\lambda g(x) = \prod_{j=1}^{d} \sqrt{\lambda_j} g( x_1/\lambda_1, \dots, x_d/ \lambda_d) $. We denote by $\mathcal{M}$, the group  of \emph{metaplectic transformations} generated by all transformations $T_a$, $M_b$, and $D_\lambda$.  A Gabor system generated by $g$ and $\Lambda \subset \mathbb{R}^d \times \mathbb{R}^d $ is $\mathcal{G}(g, \Lambda)=\{M_{b}T_{a} g: (a, b) \in \Lambda\}$. If $a_1,b_1 \in [0,\infty)$, $\mathcal{G}(g,  \mathbb{Z}^da_1\times \mathbb{Z}^db_1)$ is \emph{a rectangular} Gabor system that will be  denoted by $\mathcal{G}(g,a_1, b_1)$. If $a,b \in [0,\infty)^d$, $\mathcal{G}(g,  \mathbb{Z}a\times \mathbb{Z}b)$ is \emph{a general rectangular} Gabor system that will be  denoted by $\mathcal{G}(g,a, b)$.

One of the  most important problems in frame theory is to build examples of Bessel sequences, frames, and Riesz bases. Related to this problem is to find simple criteria that single out, in a class of sequences, Bessel sequences, frames, or Riesz bases. Several results, mostly known as density theorems, were obtained for the class of Gabor systems in $L^2(\mathbb{R})$ \cite{Hei1,Jan1,Jan2,Lyu,Sei}. Some of theses results were extended to higher dimensions using the following theorem \cite{Bou1}.
\begin{thm}\label{frame2}
The tensor product of $d$ sequences is a frame (Riesz basis)  if and only if each component of this product is a frame (Riesz basis).
\end{thm}

This paper proposes a sort of generalization of Theorem \ref{frame2}  by considering sequences in $\mathcal{H}= \bigotimes_{j=1}^{d} \mathcal{H}_j$ having the form
\begin{eqnarray}\label{decomp1}
f_{n_1,\dots, n_d} = \sum_{k=1}^{r} \bigotimes_{j=1}^{d} f_{j,k,n_j},
\end{eqnarray}
where, for each $j=1,\dots,d$,  $\big \{ \{f_{j,1,n}\},\dots \{f_{j,r,n}\} \big\}$ is a set of linearly independent sequences in $\mathcal{H}_j$, i.e., $ \{f_{n_1,\dots, n_d} \}$ is a minimal sum of tensor products of sequences. Our main theorem is the following.

\begin{thm}\label{Main} Let $d >1$.  Let $\mathcal{F}$ be a sequence in $\mathcal{H} $ defined by \eqref{decomp1}.
\begin{enumerate}[(i)]
\item $\mathcal{F}$ is a Bessel sequence in $\mathcal{H}$ if and only if $\{f_{j,k,n} \}$ is a Bessel sequence in $\mathcal{H}_j$  for each $(j,k) \in \{1,\dots , d\} \times \{1, \dots, r \}$.
\item If $\mathcal{F}$ is a frame in $\mathcal{H}$,  then  the concatenated sequence $\{ f_{j,k,n}\}_{1\leq k \leq r, n>0}$ is a frame in $\mathcal{H}_j$  for each $j=1,\dots, d$.
\end{enumerate}
\end{thm}

Section 2 includes results from \cite{Bou1,Bou2} along with other and new results in operator theory that will be used for our proofs.
In Section 3, we prove our main theorem. In Section 4, we deduce  several results on the sum of Gabor systems generated by rank one functions in $L^2(\mathbb{R}^d)$, where $d>1$. Surprisingly, the versions of some of these results for $d=1$ are  difficult to prove or disapprove.  For example,  we   prove the following density theorem for $d>1$:
\begin{examp}\label{examp}
  If $\mathcal{G}(e^{-|x|^2 } + e^{-|x|},a_1, b_1)$ is a frame in $L^2(\mathbb{R}^d)$, then $ 0 < a_1b_1 <1$.
\end{examp}
Meanwhile,  it is not clear whether the above statement holds or not for $d=1$.

\section{Operator Theory Background}
\subsection{Preliminaries and Notations}
Excluding Definition \ref{Schmidt}, all definitions and results in this subsection are included in  \cite{Kad}. We denote by $\mathcal{H}'$ the dual of $\mathcal{H}$.  If $x \in \mathcal{H}$, we denote by $x^*$ the linear form  $x^*(y)=\langle y, x \rangle$, $\forall y \in \mathcal{H}$. The transpose of the linear mapping $F: \mathcal{H} \rightarrow \mathcal{K}$ is $F^{t}: \mathcal{K}' \rightarrow \mathcal{H}'$ ,  where $F^{t}(x^{*})=x^{*}\circ F$, $ \forall x^{*} \in \mathcal{K}'$.

The Hilbert space $\mathcal{H}_1\otimes\mathcal{H}_2$ can be interpreted as the Hilbert space $\mathcal{L}_2(\mathcal{H}_1',\mathcal{H}_2)$, the space of Hilbert-Schmidt operators. The tensor product of  $F_1 \in \mathcal{B}(\mathcal{H}_1, \mathcal{K}_1)$ and  $F_2 \in \mathcal{B}(\mathcal{H}_2, \mathcal{K}_2)$ can be defined by  $F_1 \otimes F_2(H)= F_2H \ F_1^t$, for each $ H \in  \mathcal{L}_2(\mathcal{H}_1',\mathcal{H}_2)$.

\begin{defn}\label{Schmidt} We say that $F \in \mathcal{B}(\mathcal{H}_1 \otimes \mathcal{H}_2, \mathcal{K}_1 \otimes \mathcal{K}_2)$ is a \emph{finite Schmidt rank} (FSR), if it can be written in the form
\begin{eqnarray}\label{decomp}
F=\sum_{k=1}^r F_{1,k} \otimes F_{2,k},
\end{eqnarray}
where $\{ F_{1,k}\}_{k=1}^r \subset \mathcal{B}(\mathcal{H}_1, \mathcal{K}_1)$ and $\{ F_{2,k}\}_{k=1}^r \subset \mathcal{B}(\mathcal{H}_2, \mathcal{K}_2)$.
If $r$ is the minimum number such that $F$ can be written in form \eqref{decomp}, we say that $r$ is the \emph{Schmidt rank} of  $F$, denote $\mbox{rank}_\otimes(F)=r$,  and call equality  \eqref{decomp} a \emph{Schmidt decomposition} of $F$.
\end{defn}

\subsection{Schmidt Decomposition of FSR Bounded Operators} The first proposition in this subsection collects results from \cite{Bou1}.
\begin{prop}\label{prop}
Let $u=u_1\otimes u_2 \in \mathcal{H}_1 \otimes \mathcal{H}_2$ and let $v=v_1\otimes v_2 \in \mathcal{K}_1 \otimes \mathcal{K}_2$. We define the  bilinear operator
\begin{eqnarray*}
\mathcal{P}_{u, v}:  \mathcal{B}(\mathcal{H}_1\otimes\mathcal{H}_2,\mathcal{K}_1\otimes\mathcal{K}_2 )^2 &\longrightarrow & \mathcal{B}(\mathcal{H}_1,\mathcal{K}_1) \otimes  \mathcal{B}(\mathcal{H}_2,\mathcal{K}_2)\\
(F,G) &\longrightarrow & V^{v_2} F U^{u_2} \otimes  V_{v_1}G U_{u_1},
\end{eqnarray*}
where
\begin{eqnarray*}
U_{u_1}:\mathcal{H}_2  \rightarrow \mathcal{H}_1\otimes \mathcal{H}_2,  &\mbox{given by}& U_{u_1}(x_2)= u_1 \otimes x_2,\\
U^{u_2}:\mathcal{H}_1 \rightarrow \mathcal{H}_1\otimes \mathcal{H}_2,  &\mbox{given by}& U^{u_2}(x_1)= x_1 \otimes u_2, \\
V_{v_1}:\mathcal{K}_1\otimes\mathcal{K}_2\rightarrow \mathcal{K}_2,  &\mbox{given by}& V_{v_1}(H)= H(v_1^*),\\
V^{v_2}:\mathcal{K}_1\otimes\mathcal{K}_2\rightarrow \mathcal{K}_1,  &\mbox{given by}& V^{v_2}(H)=  H^t(v_2^*).
\end{eqnarray*}

\begin{enumerate}[(i)]
  \item The operators $V_{v_1} $,  $V^{v_2}$, $U_{u_1}$ , $U^{u_2}$, $\mathcal{P}_{u, v}$,   are bounded and we have
  \begin{eqnarray*}
  \| V_{v_1} \| = \|v_1\|, \quad \| V^{v_2} \| &=& \|v_2\|, \quad \| \mathcal{P}_{u, v} \| \leq \|u\| \|v\|
  \end{eqnarray*}

  \item The mapping $\mathcal{D}_{u, v}(F) = \mathcal{P}_{u, v}(F,F)$ is continuous and we have
 \begin{eqnarray*}
\| \mathcal{D}_{u, v}(F)-\mathcal{D}_{u, v}(G)\| \leq  \|u\|\| v\|\left(\|F\|+\|G\|\right)\|F-G\|.
\end{eqnarray*}
\item $\emph{rank}_\otimes (F) =1$ if and only if $\mathcal{D}_{u, v}(F)=<F(u),v>F$.
  \end{enumerate}
\end{prop}

The following results can be found in \cite{Bou2}.

\begin{thm}\label{conv}Let $\{F_N\}_{N>0} $ be a bounded sequence in $ \mathcal{B}(\mathcal{H}_1 \otimes \mathcal{H}_2, \mathcal{K}_1 \otimes \mathcal{K}_2)$ that converges in the strong operator topology to $ F \in \mathcal{B}(\mathcal{H}_1 \otimes \mathcal{H}_2, \mathcal{K}_1 \otimes \mathcal{K}_2)$. If $rank_{\otimes}(F_N) \leq r$ for each $N>0$, then $rank_{\otimes}(F) \leq r$.
\end{thm}

\begin{defn}\label{defnFMS}
We say that  $\{F_{1,k} \otimes F_{2,k}\}_{k=1}^m$ is a \emph{finite minimal system} (FMS) in $ \mathcal{B}(\mathcal{H}_1 \otimes \mathcal{H}_2, \mathcal{K}_1 \otimes \mathcal{K}_2)$ if the system $\{F_{i,k} \}_{k=1}^m$ is independent   for each $i \in \{1,2\}$.
\end{defn}

\begin{thm}\label{thmFMS} Let $ F \in \mathcal{B}(\mathcal{H}_1 \otimes \mathcal{H}_2, \mathcal{K}_1 \otimes \mathcal{K}_2)$. Equality \eqref{decomp}
is a Schmidt decomposition of $F$ if and only if  $\{ F_{1,k} \otimes F_{2,k}\}_{k=1}^r$ is a FMS.
\end{thm}

\begin{thm} \label{algorithm} Let $F \in \mathcal{B}(\mathcal{H}_1 \otimes \mathcal{H}_2, \mathcal{K}_1 \otimes \mathcal{K}_2) \setminus \{0\}$ be a FSR.  Let $(u,v)=(u_1\otimes u_2, v_1\otimes v_2) \in \mathcal{H}_1 \otimes \mathcal{H}_2 \times \mathcal{K}_1 \otimes \mathcal{K}_2$. If $\langle F(u),v \rangle = 1$, then
\begin{eqnarray*}
\emph{rank}_\otimes \left( F -\mathcal{D}_{u,v}(F) \right) = \emph{rank}_\otimes (F)-1.
\end{eqnarray*}
\end{thm}

Based on Theorem \ref{algorithm}, one can develop an algorithm to compute the Schmidt decomposition of any FSR $F \in \mathcal{B}(\mathcal{H}_1 \otimes \mathcal{H}_2, \mathcal{K}_1 \otimes \mathcal{K}_2) \setminus \{0\}$. The algorithm leads to a decomposition of $F$ after $r$ steps if and only if $\mbox{rank}_\otimes (F) =r$.

We finish this section with two new results.
\begin{cor}\label{unique} If $F \in \mathcal{B}(\mathcal{H}_1 \otimes \mathcal{H}_2, \mathcal{K}_1 \otimes \mathcal{K}_2)$ has a finite Schmidt rank $r$ and
\begin{eqnarray*}
F = \sum_{k=1}^r F_{1,k} \otimes F_{2,k} = \sum_{k=1}^r G_{1,k} \otimes G_{2,k},
\end{eqnarray*}
then, for each $ j \in\{1,2\}, span\{F_{j,1},...,F_{j,r} \}= span\{G_{j,1},...,G_{j,r}\}$.
\end{cor}
\begin{proof}  We may assume WLG that the first decomposition of $F$ is a Schmidt decomposition of $F$ obtained using an algorithm based on Theorem \ref{algorithm}. Let $u=u_1 \otimes u_2 \in \mathcal{H}_1 \otimes \mathcal{H}_2$ and $v=v_1 \otimes v_2 \in \mathcal{K}_1 \otimes \mathcal{K}_2$, for which we have $\langle F(u),v\rangle =1$ and $F_1 \otimes F_2 =\mathcal{D}_{u,v}(F)$. Using Proposition \ref{prop}, we obtain
\begin{align*}
F_{1,1} = \sum_{k=1}^r   \langle G_{2,k}(u_{2}),v_{2} \rangle G_{1,k} \quad \mbox{ and } \quad F_{2,1} = \sum_{k=1}^r  \langle G_{1,k}(u_{1}),v_1 \rangle G_{2,k}.
\end{align*}
Therefore, $F_{1,1} \in  span\{G_{1,1},...,G_{1,r}\}$ and $F_{2,1} \in span\{G_{2,1},...,G_{2,r}\}$. Similarly, computing the remaining terms of the decomposition of $F$, we can obtain  $F_{1,k} \in span\{G_{1,1},...,G_{1,r}\}$ and $F_{2,k} \in span\{G_{2,1},...,G_{2,r}\}$ for each $k=1, \dots, r$.

Consequently, for each $j=1,2$, $span\{F_{j,1},...,F_{j,r} \}$ is a vector subspace of  $span\{G_{j,1},...,G_{j,r}\}$. Meanwhile, by Theorem \ref{thmFMS}, the last two vector spaces have the same dimension $r$, and so there are equal.
\end{proof}

\begin{thm}\label{inv}  Let $F \in \mathcal{B}(\mathcal{H}_1\otimes \mathcal{H}_2)$ be FSR and have the decomposition \eqref{decomp}.
\begin{enumerate}[(i)]
  \item If $F$ has a left inverse  $L \in \mathcal{B}(\mathcal{H}_1\otimes \mathcal{H}_2)$. Then there is a set of operators $\{L_{1,k} \otimes L_{2,k}  \}_{k=1}^r  \subset \mathcal{B}(\mathcal{H}_1\otimes \mathcal{H}_2)$ for which we have
\begin{eqnarray}\label{inv1}
\forall j \in \{1,2\}, \quad \sum_{k=1}^r L_{j,k} F_{j,k} = I_j.
\end{eqnarray}
  \item If $F$ has a right inverse  $R \in \mathcal{B}(\mathcal{H}_1\otimes \mathcal{H}_2)$. Then there is a set of operators $\{R_{1,k} \otimes R_{2,k}  \}_{k=1}^r  \subset \mathcal{B}(\mathcal{H}_1\otimes \mathcal{H}_2)$ for which we have
\begin{eqnarray}\label{inv2}
\forall j \in \{1,2\}, \quad \sum_{k=1}^rF_{j,k} R_{j,k}  = I_j.
\end{eqnarray}
\end{enumerate}
\end{thm}

\begin{proof}
\begin{enumerate}[(i)]
  \item Let $u_2,v_2\in \mathcal{H}_2$ for which we have $\langle u_2, v_2\rangle=1$. For each $k \in \{1,..,r\}$, we set
$L_{1,k}= V ^{v_2}LU^{F_{2,k}(u_2)}$. For each $x_1 \in \mathcal{H}_1$, we have
\begin{eqnarray*}
\sum_{k=1}^r L_{1,k} F_{1,k}(x_1) &=& \sum_{k=1}^r  V ^{v_2}LU^{F_{2,k}(u_2)} F_{1,k}(x_1) \\
&=& \sum_{k=1}^r  V ^{v_2}L(F_{1,k}(x_1) \otimes F_{2,k}(u_2) )\\
&=& V^{v_2}(x_1 \otimes u_2)= \langle u_2,v_2\rangle x_1 = x_1.
\end{eqnarray*}

Similarly, we can prove the other equality in \eqref{inv1} by letting   $u_1,v_1\in \mathcal{H}_1$ for which we have $\langle u_1, v_1\rangle=1$ and setting, for each $k \in \{1,..,r\}$, $L_{2,k}= V _{v_1}L U_{F_{1,k}(u_1)}$.
  \item This Statement is a consequence of Statement (ii) using adjoint operators.
\end{enumerate}
\end{proof}

\section{The Tensor Product of Sequences}

The following proposition includes some elementary properties of sequences in a separable Hilbert space.
\begin{prop}\label{sequences}
Let $\{e_n\}_{n >0}$ be an ONB of $\mathcal{H}$. For a sequence $\mathcal{F} = \{f_n\}_{n >0}$ in $\mathcal{H}$, we define
\begin{eqnarray}\label{operator}
  \forall f \in D(\mathcal{H}), \quad  F(f)=\sum_{n>0} \langle f, f_n \rangle e_n,
  \end{eqnarray}
where $D(\mathcal{H})=\{ f \in  \mathcal{H}: F(f) \in \mathcal{H}\}$. We write $F= Ass(\{e_n\}_{n >0},\mathcal{F})$ and read $F$  is associated to $\{e_n\}_{n >0}$ and $\mathcal{F}$.
\begin{enumerate}[(i)]
  \item $\mathcal{F}$ is a Bessel sequence if and only if $F$ is a bounded operator. In addition, $\|F \|^2$  is a Bessel bound for $\mathcal{F}$.
  \item $\mathcal{F}$ is a frame if and only if $F$ is a bounded operator and has a left inverse $L \in \mathcal{B}(\mathcal{H})$. In addition, $1/\|L \|^2$ is a lower frame bound for $\mathcal{F}$.
  \item $\mathcal{F}$ is a Riesz basis if and only if $F$ is an invertible bounded operator.
\end{enumerate}
\end{prop}

In Equality \eqref{operator}, the convergence is meant to be conditional. Therefore, the set of positive integers can be replaced with any other countable set of indices.

$F^*(f)=\sum_{n>0}  \langle f, e_n \rangle f_n$ is the so  called \emph{preframe operator}  \cite{Cas}. We then have   $\{F^*(e_n)\}= \{f_n \}$ and $F^*F=S$, where $S$ is the so called \emph{frame operator}.

From now on, $ \{e_{j,n}\}_{n >0}$ will be be an ONB of $\mathcal{H}_j$ for each $j \in \{1,\dots,d\}$.

\begin{thm}\label{Bessel}Let $\mathcal{F} =\{f_{m,n} = f_{1,1,m} \otimes f_{2,1,n} + \dots + f_{1,r,m} \otimes f_{2,r,n}\}_{m,n >0}$ in $\mathcal{H}= \mathcal{H}_1 \otimes \mathcal{H}_2$, where, for each $j \in \{1,2\}$,
$\big\{ \{f_{j,1,n} \}_{n >0}, \dots , \{f_{j,r,n}\}_{n >0} \big \}$ is a set of linearly independent   sequences in $\mathcal{H}_j$.  Let $F= Ass(\{e_{1,m} \otimes e_{2,n}\}_{m,n >0},\mathcal{F})$ and let $F_{j,k}= Ass(\{e_{j,n}\},\{f_{j,k,n} \})$, for each $(j,k) \in \{1,2\} \times \{1, \dots, r \}$.
\begin{enumerate}[(i)]
\item $\mathcal{F}$ is a Bessel sequence in $\mathcal{H}$ if and only if $\{f_{j,k,n} \}_{n >0}$ is a Bessel sequence in $\mathcal{H}_j$  for each $(j,k) \in \{1,2\} \times \{1, \dots, r \}$.
\item If $\mathcal{F}$ is a frame,  then  there is a set of operators $\{L_{1,k} \otimes L_{2,k}  \}_{k=1}^r  \subset \mathcal{B}(\mathcal{H})$ for which we have Equality\eqref{inv1}. In addition, the concatenated sequence $\{ f_{j,k,m}\}_{1\leq k \leq r, m>0}$ is a frame in $\mathcal{H}_j$  for each $j=1,2$.
 \item If $\mathcal{F}$ is a Riesz basis in $\mathcal{H}_1 \otimes \mathcal{H}_2$,  then, in addition to the conclusion in statement (ii),  there is a set of operators $\{R_{1,k} \otimes R_{2,k}  \}_{k=1}^r  \subset \mathcal{B}(\mathcal{H})$ for which we have Equality \eqref{inv2}.
\end{enumerate}
\end{thm}

\begin{proof}
\begin{enumerate}[(i)]
\item Obviously, if $\{f_{j,k,n} \}_{n >0}$ is a Bessel sequence in $\mathcal{H}_j$  for each $(j,k) \in \{1,2\} \times \{1, \dots, r \}$, then
$\mathcal{F}$ is a Bessel sequence in $\mathcal{H}_1 \otimes \mathcal{H}_2$. For the converse, we define the sequence
\begin{eqnarray}\label{3.1}
  \forall N>0, \quad F_{N} = \sum_{k=1}^r F^{N}_{1,k} \otimes F^{N}_{2,k},
\end{eqnarray}
where, for each $(j,k) \in \{1,2\} \times \{1, \dots, s \}$ and each $f \in \mathcal{H}_j$,
\begin{eqnarray*}
F^{N}_{j,k}(f) = \sum_{n=0}^N \langle f, f_{j,k,n} \rangle  e_{j,n}.
\end{eqnarray*}

We have $F,  F_N \in \mathcal{B}(\mathcal{H}_1  \otimes \mathcal{H}_2) $, $\|F_N \| \leq \|F\|$ for each $N>0$,  the sequence $ \{F_{N}\}_{N >0}$ strongly converges to $F$,  and $rank_{\otimes}(F_N) \leq r$ for each $N>0$. Using Theorem \ref{conv}, we then obtain $rank_{\otimes}(F)=s \leq r$, and so
\begin{eqnarray*}
  F&=&\sum_{k=1}^s  F_{1,k} \otimes F_{2,k},
\end{eqnarray*}
where, for each $j \in \{1,2\}$,  $\{ F_{j,1}, \dots F_{j,s} \}$ is a set of linearly independent bounded operators.

For each $(j,k) \in \{1,2\} \times \{1, \dots, s \}$, we define the sequence $g_{j,k,n }= F_{j,1}^*(e_{j,n})$. Therefore, for each $j \in \{1,2\}$,  $\big \{ \{g_{j,1,n }\}_{n>0}, \dots, \{g_{j,s,n }\}_{n>0} \big\} $ is a set of independent Bessel sequences. In addition, we have
\begin{eqnarray}\label{3.2}
  \forall N>0, \quad F_{N} = \sum_{k=1}^s G^{N}_{1,k} \otimes G^{N}_{2,k},
\end{eqnarray}
where, for each $(j,k) \in \{1,2\} \times \{1, \dots, s \}$ and each $f \in \mathcal{H}_1  \otimes \mathcal{H}_2$,
\begin{eqnarray*}
G^{N}_{j,k}(f)=\sum_{n=0}^N \langle f, g_{j,k,n} \rangle e_{j,n}.
\end{eqnarray*}

We can choose $N>0$  big enough to have independent each one of the sets of sequences $\big \{\{f_{j,1,n }\}_{n=1}^N, \dots, \{f_{j,r,n }\}_{n=1}^N \big \} $  and $\big \{ \{g_{j,1,n }\}_{n=1}^N, \dots, \{g_{j,s,n }\}_{n=1}^N \big \} $, where $j \in \{1,2\}$. Hence, Equalities \eqref{3.1} and \eqref{3.2} become Schmidt decompositions, and so, using Theorem \ref{thmFMS}, $r=s$.

Using Corollary \ref{unique},  the sequence  $\{f_{j,k,n }\}_{n=1}^N$ is a linear combination of the sequences $\{g_{j,1,n }\}_{n=1}^N, \dots, \{g_{j,r,n }\}_{n=1}^N $, for each $(j,k) \in \{1,2\} \times \{1, \dots, r \}$.  Since $N$ can be chosen as big as we want, we conclude  the sequence  $\{f_{j,k,n }\}_{n>0}$ is a linear combination of the sequences $\{g_{j,1,n }\}_{n>0}, \dots, \{g_{j,r,n }\}_{n>0} $, for each $(j,k) \in \{1,2\} \times \{1, \dots, r \}$. Consequently, we obtain the desired result.

\item  Using (i), $\{f_{j,k,n} \}$ is a Bessel sequence in $\mathcal{H}_j$  for each $(j,k) \in \{1,2\} \times \{1, \dots, r \}$, and so the concatenated sequence $\{ f_{j,k,n}\}_{1\leq k \leq r, n>0}$ is a Bessel sequence in $\mathcal{H}_j$  for each $j=1,2$. By Proposition \ref{sequences}, the operator
\begin{eqnarray*}
F=\sum_{k=1}^r F_{1,k} \otimes F_{2,k}
\end{eqnarray*}
is left invertible, and so, by Theorem \ref{inv}, there is a set of operators $\{L_{1,k} \otimes L_{2,k}  \}_{k=1}^r  \subset \mathcal{B}(\mathcal{H}_1\otimes \mathcal{H}_2)$ for which we have Equality \eqref{inv1}. Consequently, for each $j=1,2$, we have
\begin{eqnarray*}
\forall f \in \mathcal{H}_j, \quad \sum_{k=1}^r \|L_{j,k} \| \left( \sum_{n>0} |\langle f, f_{j,k,n} \rangle|^2  \right)^{1/2} \geq \|f\|,
\end{eqnarray*}
and so the concatenated sequence $\{ f_{j,k,n}\}_{1\leq k \leq r, m>0}$ is a frame.
 \item This statement is obtained using Theorem \ref{inv} (ii) and the two previous statements.
\end{enumerate}
\end{proof}

Now, we can prove Theorem \ref{Main}.
\begin{proof}
\begin{enumerate}[(i)]
  \item This statement is obtained by induction using Theorem \ref{Bessel} (i).
  \item On one hand, Statement (i) implies that $\{ f_{j,k,n}\}_{1\leq k \leq r, n>0}$ is a Bessel sequence for each $j=1,\dots,d$, and then so is the sequence
  \begin{eqnarray}\label{seq}
\bigotimes_{j=1}^{d}  \{ f_{j,k,n_j}\}_{1\leq k \leq r, n_j>0}.
\end{eqnarray}
  On the other hand, Theorem \ref{Bessel} implies that the sequences $\{ f_{1,k,n}\}_{1\leq k \leq r, n>0}$ and $\{  \bigotimes_{j=2}^{d} f_{j,k,n_j} \}_{1\leq k \leq r, n_2,\dots, n_d >0}$  are  frames. Hence, by Theorem \ref{frame2}, the sequence $\{  \bigotimes_{j=1}^{d} f_{j,k,n_j} \}_{1\leq k \leq r, n_2,\dots, n_d >0}$ is a frame. The fact that the last sequence is included in Sequence \ref{seq} implies that Sequence \ref{seq} is also a frame.

 Using Theorem \ref{frame2} one more time, we conclude that $\{ f_{j,k,n}\}_{1\leq k \leq r, n>0}$ is a frame for each $j\in \{1,\dots,d\}$.
\end{enumerate}
\end{proof}

\begin{cor}\label{2sequences}
 Let $d >1$. For each $j \in \{1,\dots, d\}$, let $ \{f_{j,1,n} \}$ and $ \{f_{j,2,n}\}$ be  two linearly independent sequences in $\mathcal{H}_j$. Let $E_k= \big \{ \{f_{j,k,n} \} \big\}_{j=1}^{d}$, where $k=1,2$. Let $\mathcal{F}=\mathcal{F}_1 + \mathcal{F}_2$, where  $\mathcal{F}_k=\big\{\bigotimes_{j=1}^{d} f_{j,k,n} \big\}$ for each $k=1,2$.
 \begin{enumerate}[(i)]
  \item If $\mathcal{F}$ is a frame, then $\mathcal{F}_1$ or $\mathcal{F}_2$ is a frame or there is an index $i \in \{1,\dots,d\}$ such that $E_k \setminus \big \{ \{f_{i,k,n} \} \big\} $ is a set of frames for each $k=1,2$.
  \item If $\mathcal{F}$ is a Riesz basis and $E_1 \times E_2$ is a set of frames, then $\mathcal{F}_1$ or $\mathcal{F}_2$ is a Riesz basis or there is an index $i \in \{1,\dots,d\}$ such that $E_k \setminus \big \{ \{f_{i,k,n} \} \big\} $ is a set of Riesz bases for each $k=1,2$.
\end{enumerate}
 \end{cor}

\begin{proof} Let $F_{j,k}= Ass(\{e_{j,n}\},\{f_{j,k,n} \}), \forall (j,k) \in \{1, \dots , d \} \times \{1,2\}$. Therefore
\begin{eqnarray*}
\forall k=1,2, \quad F_k= Ass(\{\bigotimes_{j=1}^{d}e_{j,n_j} \},\mathcal{F}_k) =\bigotimes_{j=1}^{d} F_{j,k}
\end{eqnarray*}
and $F= Ass(\{\bigotimes_{j=1}^{d}e_{j,n_j} \},\mathcal{F})= F_1 + F_2$.

We assume that $\mathcal{F}$ is a frame. Therefore, by Theorem \ref{Bessel}, $ \{f_{j,k,n} \}$ is a Bessel sequence for each $(j,k) \in \{1, \dots, d \} \times \{1, 2 \}$.
\begin{enumerate}[(i)]
  \item a. We first prove the statement for $d=2$. Assume that $ \{f_{2,2,m} \}$ is not a frame. Therefore, there is a sequence $\{x_{2,n}\}$ in $\mathcal{H}_2$ such that $\| x_{2,n}\| =1, \forall n>0$ and the sequence $\{F_{2,1}(x_{2,n})\}_{n>0}$ converges to zero.  The fact that $\mathcal{F}$ is a frame implies the inequality
\begin{eqnarray*}
\forall x_1 \in \mathcal{H}_1, \quad  \|F_{1,1}(x_1) \otimes  F_{2,1}( x_{2,n}) +   F_{1,2}(x_1) \otimes  F_{2,2}(x _{2,n}) \| \geq \sqrt{A}\| x_1 \|,
\end{eqnarray*}
where $A$ is the lower frame bound of $\mathcal{F}$. Therefore
\begin{eqnarray*}
\forall x_1 \in \mathcal{H}_1, \quad  \| F_{1,1}(x_1)\| \geq \frac{\sqrt{A}\| x_1 \|}{\| F_{2,1} \|},
\end{eqnarray*}
and so $ \{f_{1,1,m} \}_{m >0}$ is a frame. Similarly, if $ \{f_{1,2,m} \}_{m >0}$ is not a frame, then $ \{f_{2,1,m} \}_{m >0}$ is a frame. Hence,  we obtain the desired result for $d=2$.

b. We now assume that $d>2$. Assume that $ \{f_{i,2,m} \}$ is not a frame for some $i \in \{1, \dots, d\}$. If $i <d$, then, by Theorem \ref{frame2}, the  sequence $\bigotimes_{j=1}^{i} f_{j,2,n}$ is not a frame, and so using (a), the  sequence $\bigotimes_{j=i+1}^{d} f_{j,1,n}$ is a frame. And if $i >1$, then, by Theorem \ref{frame2}, the  sequence $\bigotimes_{j=i}^{d} f_{j,2,n}$ is not a frame, and so using (a), the  sequence $\bigotimes_{j=1}^{i-1} f_{j,1,n}$ is a frame. In any case, by Theorem \ref{frame2}, we conclude that $f_{j,1,n}$ is a frame for each $j \in \{1, \dots, d \} \setminus \{i\}$.
Hence, we can obtain the desired result by using Theorem \ref{frame2} once again.

\item Since $\mathcal{F}$ is a Riesz basis, the sequence $\mathcal{F}^{*}=\big \{  \bigotimes_{j=1}^{d} f_{1,1,n_j}^{*} + \bigotimes_{j=1}^{d} f_{1,2,n_j}^{*} \big\}$ is a frame, where $f_{j,k,n}^{*}= F_{j,k}(e_{j,n})$ for each $(j,k) \in \{1, \dots , d \} \times \{1,2\}$. Using (ii) with $\mathcal{F}^{*}$ and the fact that $E_1 \times E_2$ is a set of frames, we can obtain what we want.
\end{enumerate}
\end{proof}

\begin{rem} Assume that $\mathcal{F}=\bigotimes_{j=1}^{d} f_{j,1,n} +\bigotimes_{j=1}^{d} f_{j,2,n}$ is a frame. Using Theorem \ref{frame2} and Corollary \ref{2sequences}, the set of sequences $E_1 \cup E_2$ includes at least $d$ frames.
\end{rem}

\section{Gabor Systems generated by  Finite Rank Windows}

Using Theorem \ref{Main}, we can deduce the following proposition for Gabor systems.

\begin{prop}\label{BesselGabor} Let $d>1$. Let $g$ be a rank $r$ function in $L^2(\mathbb{R}^{d})$, i.e.,
\begin{eqnarray}\label{function}
g = \sum_{k=1}^{r} \bigotimes_{j=1}^{d}g_{j,k},
\end{eqnarray}
where, $\forall j \in \{1, \dots, d\}$,   $ \{g_{j,1}, \dots, g_{j,d} \}$ is a set of independent functions in $L^2(\mathbb{R})$. $\mathcal{G}(g,\prod _{j=1}^{d}\Lambda_{j})$ is a Bessel sequence in $L^2(\mathbb{R}^{d})$ if and only if $\mathcal{G}(g_{j,k},\Lambda_{j})$ is a Bessel sequence in $L^2(\mathbb{R})$ for each $(j,k) \in \{1,\dots,d\} \times \{1,\dots,r\}$.
\end{prop}

\begin{rem}\label{rBessel}
   If $\mathcal{F}$ is a sum  of $r$ independent sequences in $\mathcal{H}$,  the fact that $\mathcal{F}$  is a Bessel sequence does not necessary implies that each term in $\mathcal{F}$ is a Bessel sequence.  In addition, if $\mathcal{F}$ is assumed to be a Bessel sequence, it is extremely difficult to prove whether each term in $\mathcal{F}$ is a Bessel sequence or not. In particular, the one dimensional version of Proposition \ref{BesselGabor} does not always hold.
\end{rem}

The following density theorem for general rectangular Gabor systems is a particular case of the density theorem for lattices that is proved in \cite{Chr}.

\begin{thm}\label{Density} Let $g \in L^2(\mathbb{R}^d)$ and   let $(a,b) \in [0, \infty)^{2d} \setminus \{0\} $.
\begin{enumerate}[(i)]
  \item If $\mathcal{G}(g,a, b)$ is a frame, then $0< \prod_{j=1}^d a_jb_j \leq 1$.
  \item $\mathcal{G}(g,a, b)$ is a Riesz basis if and only if $\mathcal{G}(g,a, b)$ is a frame and $ \prod_{j=1}^d a_jb_j = 1$.
\end{enumerate}
\end{thm}

For the rest of this section, we assume that $(a,b) \in (0, \infty)^{d} \times (0, \infty)^{d}$.

On the one hand, if $ \prod_{j=1}^d a_jb_j \leq 1$, then  there is $g \in L^2(\mathbb{R}^d)$ for which $\mathcal{G}(g,a, b)$ is a frame \cite{Bek}. On the other hand, Theorem \ref{frame2} implies the following.
\begin{quote}
  Let $g$ be a one rank function in $L^2(\mathbb{R}^{d})$. If $\mathcal{G}(g,a, b)$ is a frame, then $ a_jb_j \leq 1,\forall j=1,\dots,d$.
\end{quote}
The last result and Proposition \ref{BesselGabor} let us conjecture for finite rank functions $g \in L^2(\mathbb{R}^{d})$:  $\mathcal{G}(g,a, b)$ is a frame implies sharper than the condition $ \prod_{j=1}^d a_jb_j \leq 1$ or maybe the sharpest condition that is $ a_jb_j \leq 1,\forall j=1,\dots,d$.

The following proposition supports our conjecture in case the rank of $g$ is $2$.
\begin{prop}\label{Gabor2} Let $d>1$. For each $j \in \{1,\dots,d \}$, let $g_{j,1}, g_{j,2}\in L^2(\mathbb{R})$ be two independent functions. Assume that $\mathcal{G}(g=g_1+g_2,a, b)$ is a frame, where
$ g_k=\bigotimes_{j=1}^{d} g_{j,k}$ for $k=1,2$.
\begin{enumerate}[(i)]
  \item $\mathcal{G}(g_1,a, b)$ or $\mathcal{G}(g_2,a, b)$ is a frame or  there is an index $i \in \{1, \dots, d \}$ such that $\big \{ \mathcal{G}(g_{j,k},a, b) \big \}_{j=1, j\neq i}^d$ is a set of frames for each $k=1,2$. In any case, the condition $a_jb_j \leq 1$ holds for at least $(d-1)$ indices $j \in \{1, \dots, d \}$.
  \item If $\mathcal{G}(g,a, b)$ is a Riesz basis, then $\mathcal{G}(g_1,a, b)$ or $\mathcal{G}(g_2,a, b)$ is a Riesz basis or there is an index $i \in \{1, \dots, d \}$ such that $\big \{ \mathcal{G}(g_{j,k},a, b) \big \}_{j=1, j\neq i}^d$ is a set of Riesz bases for each $k=1,2$.
\end{enumerate}
\end{prop}

\begin{proof}
\begin{enumerate}[(i)]
  \item The first part of this Statement is a consequence of Corollary \ref{2sequences} (i);  and the second part is obtained using Theorem \ref{Density} (i).
  \item Using Statement (i) and Theorem \ref{Density}.(ii), we can obtain Statement (ii).
\end{enumerate}

\end{proof}

In the followings, we consider finite rank functions $g \in L^2(\mathbb{R}^d)$ having the form
\begin{eqnarray}\label{Gaboreq}
g=\sum_{k=1}^{r} \bigotimes_{j=1}^{d} M_{\beta_{j,k}} T_{\alpha_{j,k}} g_j,
\end{eqnarray}
where, $g_1, \dots, g_d \in L^2(\mathbb{R}) \setminus \{0\}$ and, for each $j\in \{1,\dots,d\}$,  $\{(\alpha_{j,k}, \beta_{j,k})\}_{k=1}^{r}$  is a set of $r$ different  nonzero ordered pairs.

\begin{prop}\label{Gabor} Let $d>1$. Let $g \in L^2(\mathbb{R}^d)$ be defined by Equality \eqref{Gaboreq}. For each $(j,k) \in \{1,\dots,d\} \times \{1,\dots,r\}$, assume that $\alpha_{j,k} / a_j, \beta_{j,k} / b_j  \in \mathbb{Z}$. Assume that  $\mathcal{G}(g,a, b)$ is a frame in $L^2(\mathbb{R}^d)$.
\begin{enumerate}[(i)]
  \item For each $j\in \{1,\dots,d\}$, $\mathcal{G}(g_j,a_j, b_j)$ is a frame in $L^2(\mathbb{R})$, and so $a_jb_j \leq 1$.
  \item If $\mathcal{G}(g,a, b)$ is a Riesz basis, then $\mathcal{G}(g_j,a_j, b_j)$ is a Riesz basis in $L^2(\mathbb{R})$ for each $j\in \{1,\dots,d\}$.
\end{enumerate}
\end{prop}

\begin{proof}
\begin{enumerate}[(i)]
  \item Let $j\in \{1,\dots,d\}$. Using a proven partial result of the HRT conjecture \cite{Hei2},  $\{M_{\beta_{j,k}} T_{\alpha_{j,k}} g_j\}_{j=1}^{r}$ is an independent set of square integrable functions. Therefore, by Proposition \ref{BesselGabor}, $\mathcal{G}(M_{\beta_{j,k}} T_{\alpha_{j,k}} g_j,a_j, b_j)$ is a Bessel sequence, and then so is $\mathcal{G}(g_j,a_j, b_j)$.

  Using Theorem \ref{Main}, $\bigsqcup_{k=1}^r \mathcal{G}\left( M_{\beta_{j,k}} T_{\alpha_{j,k}} g_j,a_j,b_j \right)$ is a frame, for which we let $A_j$ be the lower bound. Therefore, $\forall f \in L^2(\mathbb{R})$, we have
\begin{eqnarray*}
\sum_{m,n \in \mathbb{Z}} \sum_{i=1}^r \mid \langle f,M_{nb_j} T_{ma_j} M_{\beta_{j,k}} T_{\alpha_{j,k}} g_j \rangle \mid^2 &\geq& A_j \| f \|^2\\
\sum_{m,n \in \mathbb{Z}} r \mid \langle f,M_{nb_j} T_{ma_j} g_j \rangle \mid^2 &\geq& A_j \| f \|^2.
\end{eqnarray*}
The last inequality and the fact that $\mathcal{G}(g_j,a_j, b_j)$ is a Bessel sequence imply that $\mathcal{G}(g_j,a_j, b_j)$ is a frame. Using Theorem \ref{Density}, we then obtain $a_jb_j \leq 1$.

 \item On the one hand, if $\mathcal{G}(g,a, b)$ is a Riesz basis, then, by Theorem \ref{Density}, $a_jb_j = 1$ for each $j \in \{1, \dots, d\}$. On the other hand, Statement (i) implies that $\mathcal{G}(g_j,a_j, b_j)$ is a frame for each $j\in \{1,\dots,d\}$. Therefore, using Theorem \ref{Density} again, we deduce that $\mathcal{G}(g_j,a_j, b_j)$ is a Riesz basis  for each $j\in \{1,\dots,d\}$.
\end{enumerate}
\end{proof}

The results in the following lemma seems interesting and its proof is not difficult, but we could not find references for these results.

\begin{lem}\label{lGabor}  Let $g \in L^2(\mathbb{R}^d) \setminus \{0\}$.
\begin{enumerate}[(i)]
\item If $\mathcal{G}(g,a, b)$ be a frame with the frame bounds $A$ and $B$, then, for all positive integers $u_1,v_1, \dots u_d,v_d $,  $\mathcal{G}\big(g,(a_1/u_1,\dots a_d/u_d), (b_1/v_1, \dots b_d/v_d) \big)$ is a frame with the frame bounds $u_1 \dots u_dA$ and $v_1 \dots v_d B$.
\item If $\mathcal{G}(g,a, b)$ is a Bessel sequence, then so is $\mathcal{G}(g, \langle u, a\rangle, \langle v , b\rangle)$ for all $u,v \in \mathbb{Q}  ^{d} $ such that $u_1,v_1, \dots u_d,v_d >0$.
\end{enumerate}
\end{lem}
\begin{proof}
 For the sake of notation simplicity, we prove the lemma for $d=1$.
 \begin{enumerate}[(i)]
  \item  Assume that $\mathcal{G}(g,a, b)$ is a frame in $L^2(\mathbb{R})$ with the frame bounds $A$ and $B$ and let $(u,v)$ be a pair of positive integers. Hence,  $\mathcal{G}(M_{jb/v}T_{ia/u}g,a, b)$ is a frame with the frame bounds $A$ and $B$, $\forall(i,j) \in \{0,\dots,u-1 \} \times  \{0,\dots,v -1\} $, and so, for each $f \in L^2(\mathbb{R})$, we have
\begin{eqnarray*}
 A\|f\|^2 \leq \sum_{m,n \in \mathbb{Z}}  \mid \langle f, M_{(nv+j)b/v}T_{(m u+i)a/u} g\rangle \mid^2 &\leq & B \|f\|^2.
\end{eqnarray*}
Hence,
\begin{eqnarray*}
uv A \|f\|^2 \leq \sum_{m,n \in \mathbb{Z}} \sum_{i,j=0}^{u-1,v-1}  \mid \langle f, M_{(nv+j)b/v}T_{(m u+i)a/u} g\rangle \mid^2 \leq uv B\|f\|^2.
\end{eqnarray*}
The last inequality proves that $\mathcal{G}(g,a/u, b/v)$ is a frame.

\item Let  $u_1,u_2,v_1,v_2$ be positive integers. As we did in Statement (i), if $\mathcal{G}(g,a, b)$ is a Bessel sequence, then so is $\mathcal{G}(g,a/u_2, b/v_2)$  and this obviously implies that $\mathcal{G}(g,au_1/u_2, bv_1/v_2)$ is a Bessel sequence.
\end{enumerate}
\end{proof}

\begin{rem} Some interesting consequences of Lemma \ref{lGabor} are as follows.
\begin{enumerate}[(i)]
\item If there is  $(a_0,b_0) \in (0,\infty)^{2d}$ for which  $\mathcal{G}(g,a_0,b_0)$ is a Bessel sequence, then the set of  $(a,b)$ for which  $\mathcal{G}(g,a, b)$ is a Bessel sequence is dense in $(0,\infty)^{2d}$.
\item If $\mathcal{G}(g,a, b)$ is a Bessel sequence for all $(a,b)$ in an open set in $(0,\infty)^{2d}$, then $\mathcal{G}(g,a, b)$ is a a Bessel sequence for all $(a,b) \in  (0,\infty)^{2d}$.
  \item If $\mathcal{G}(g,a, b)$ is a frame for all $(a,b)$ in an open set in $(0,\infty)^{2d}$, then $\mathcal{G}(g,a, b)$ is a frame for all $(a,b) \in O \cap (0,\infty)^{2d}$, where $O$ is a neighborhood of $0$.
\end{enumerate}
\end{rem}

\begin{thm}\label{tGabor} Let $d>1$. Let $g \in L^2(\mathbb{R}^d)$ be defined by Equality \eqref{Gaboreq}. For each $j \in \{1,\dots,d\} $, assume that $\{M_{\beta_{j,k}} T_{\alpha_{j,k}} g_j\}_{k=1}^{r}$ is an independent set of square integrable functions. Assume that  $\mathcal{G}(g,a, b)$ is a frame in $L^2(\mathbb{R}^d)$ and $(a,b) \in \mathbb{Q}^{2d}$.
\begin{enumerate}[(i)]
  \item If $\{\alpha_{j,k}, \beta_{j,k}\}_{j=1,k=1}^{j=d,k=r} \subset \mathbb{Q}$, then there is $M>0$ such that $a_jb_j \leq M$  for each $j\in \{1,\dots,d\}$.
  \item If $\{\alpha_{j,k}, \beta_{j,k}\}_{j=1,k=1}^{j=d,k=r} \subset \mathbb{Z}$, then $a_jb_j \leq 1$ for each $j\in \{1,\dots,d\}$.
\end{enumerate}
\end{thm}

\begin{proof} We denote by $q(x)$ the positive denominator of each rational number $x$ written in its simplest form. For each $j=1, \dots, d$, let $\alpha_{j} = \prod_{k=1}^{r} q(\alpha_{j,k})$ and $\beta_{j} = \prod_{k=1}^{r} q(\beta_{j,k})$. Using Lemma \ref{lGabor}, we can assume without lost of generality that $1/a_j,1/b_j \in \mathbb{Z} $.

\begin{enumerate}[(i)]
  \item Let $j \in \{1, \dots, d\}$. Using Theorem \ref{Main}, $\bigsqcup_{k=1}^r \mathcal{G}\left( M_{\beta_{j,k}} T_{\alpha_{j,k}} g_j,a_j,b_j \right)$ is a frame, for which we let $A_j$ be a lower bound. Therefore, $\forall f \in L^2(\mathbb{R})$, we have
\begin{eqnarray}\label{rat}
\sum_{m,n \in \mathbb{Z}} \sum_{k=1}^r \mid \langle f,M_{nb_j} T_{ma_j} M_{\beta_{j,k}} T_{\alpha_{j,k}} g_j \rangle \mid^2 \geq A_j \| f \|^2 \nonumber\\
\sum_{m,n \in \mathbb{Z}} \sum_{k=1}^r \mid \langle f,M_{[n\beta_{j} + \frac{\beta_{j}\beta_{j,k}}{b_j} ]\frac{b_j}{\beta_{j}}} T_{[m\alpha_{j} + \frac{\alpha_{j} \alpha_{j,k}}{a_{j}} ]\frac{a_j}{\alpha_{j}}} g_j \rangle \mid^2\geq A_j \| f \|^2 \nonumber \\
\sum_{m,n \in \mathbb{Z}} r \mid \langle f,M_{n\frac{b_j}{\beta_{j}}} T_{m\frac{a_j}{\alpha_{j}}} g_j \rangle \mid^2\geq A_j \| f \|^2
\end{eqnarray}

By Proposition \ref{BesselGabor}, for each $k=1,\dots,r$,  $\mathcal{G}(M_{\beta_{j,k}} T_{\alpha_{j,k}} g_j,a_j, b_j)$ is a Bessel sequence, and then so is $\mathcal{G}(g_j,a_j, b_j)$. Therefore, by Lemma \ref{lGabor}, $\mathcal{G}(g_j,a_j/\alpha_{j}, b_j/\beta_{j})$ is a Bessel sequence, and so, using Inequality \eqref{rat}, $\mathcal{G}(g_j,a_j/\alpha_{j}, b_j/\beta_{j})$ is a frame. Thus, using Theorem \ref{Density}, we obtain $a_jb_j \leq \alpha_{j}\beta_{j}$.

Consequently, the desired conclusion is obtained with $M = \max \{\alpha_{j}\beta_{j}\}_{j=1}^{d}$.
  \item In this case, $\alpha_{j}= \beta_{j} = 1$ for each $j\in \{1,\dots,d\}$, and so $M=1$.
\end{enumerate}
\end{proof}

\begin{rem} If $g$ is defined by Equality \eqref{Gaboreq}, then $\mathcal{G}(g,a, b)$ is, in fact, a sum  of Gabor systems:
\begin{eqnarray*}
\mathcal{G}(g,a, b)=\sum_{k=1}^{r} \mathcal{G}(g_k,a, b), \mbox{ where } g_k=\bigotimes_{j=1}^{d} M_{\beta_{j,k}} T_{\alpha_{j,k}} g_j.
\end{eqnarray*}
By Theorem \ref{frame2}, we can incidently prove under the conditions of Proposition \ref{Gabor}(i) or Theorem \ref{tGabor}(ii) that $\mathcal{G}(g_k,a, b)$ is a frame, $\forall k=1,\dots,r$, whenever $\mathcal{G}(g,a, b)$ is assumed to be a frame. We can also prove under the conditions of  Proposition \ref{Gabor}(ii) that $\mathcal{G}(g_k,a, b)$  is a Riesz basis, $\forall k=1,\dots,r$, whenever $\mathcal{G}(g,a, b)$ is assumed to be a Riesz basis. Meanwhile, the one dimensional version of the aforementioned results are are difficult to prove or disapprove.  This is somehow surprising, because results for Gabor systems are traditionally easier for the one dimensional case.
\end{rem}

  In general, results on  the sum of sequences are  extremely difficult to obtain even if the sum includes only two terms that related to each other \cite{Obe}. In the followings, we  state results on the sum of two Gabor systems in case the sum is defined by Equality \eqref{Gaboreq}. For that, we use the following proposition, for which the proof is straightforwardly similar to the proof for $d=1$ that is published in \cite{Naj}.

\begin{prop}\label{Naj}  Let $a_1, \dots,a_d, b_1, \dots ,b_d >0$ and let $(\alpha_1,\beta_1), \dots,(\alpha_d,\beta_d) \in \mathbb{R}^2 \setminus \{0\} $ such that $\alpha_j b_j, \beta_j a_j  \in \mathbb{Z}, \forall j=1,\dots,d$. Let $c \in \mathbb{C}$ with $|c|=1$. If $g \in L^2(\mathbb{R}^d)$
 and $\mathcal{G}(g,a, b)$ is a frame, then $\mathcal{G}(g+c M_{\beta} T_{\alpha}g,a, b) $ is not a frame.
\end{prop}

\begin{thm}\label{Sum} Let $d>1$.   Let $a_1, \dots,a_d, b_1, \dots ,b_d >0$ and let $(\alpha_1,\beta_1), \dots,(\alpha_d,\beta_d) \in \mathbb{R}^2 \setminus \{0\} $  such that $\alpha_j b_j, \beta_j a_j, \alpha_j/a_j, \beta_j/b_j \in \mathbb{Z}$. Let $c \in \mathbb{C}$ with $|c|=1$. For each function $g= \bigotimes_{j=1}^{d} g_j \in L^2(\mathbb{R}^d)$,  $\mathcal{G}(g+c M_{\beta} T_{\alpha}g,a, b) $ is not a frame.
\end{thm}

\begin{proof} Assume that $\mathcal{G}(g+c M_{\beta} T_{\alpha}g,a, b) $ is a frame. Using Proposition \ref{Gabor}, $\mathcal{G}(g_j,a_j, b_j) $ is a frame for each $j=1,\dots, d$, and so, using Theorem \ref{frame2}, $\mathcal{G}(g,a, b) $ is then a frame. The last fact and Proposition \ref{Naj} imply the contradiction $\mathcal{G}(g+c M_{\beta} T_{\alpha}g,a, b) $ is a not frame.
\end{proof}

\begin{cor}\label{cSum} Let $g= \bigotimes_{j=1}^{d} g_j  + \bigotimes_{j=1}^{d} T_{1}g_j$, where $g_1,\dots,g_d \in L^2(\mathbb{R})$.  If $\mathcal{G}(g, a_1, b_1 )$ is a frame, then  $a_1b_1$ is an irrational number.
\end{cor}
\begin{proof} By Theorem \ref{Sum}, $\mathcal{G}(g, 1/n, 1 )$ is not a frame for all integer $n>0$. Therefore, using Lemma \ref{lGabor}, $\mathcal{G}(g, m/n, 1 )$ is not a frame for all integer $m,n>0$. Hence, if $a_1b_1 \in \mathbb{Q}$, then $\mathcal{G}(g, a_1b_1, 1 )$ is not a frame , and so $\mathcal{G}(g, a_1, b_1 )$ is not a frame.
\end{proof}

If $h  \in  M^{1}(\mathbb{R})$, then  $\{(a_1,b_1): \mathcal{G}(h,a_1,b_1) \mbox{ is a frame} \}$ is an open subset of $\mathbb{R}^2$ \cite{Fei}. However, for each $d>1$, using Corollary \ref{cSum}, we can find $g \in  M^{1}(\mathbb{R}^d)$ such that $\{(a_1,b_1): \mathcal{G}(g,a_1,b_1) \mbox{ is a frame} \}$ is not an open subset of $\mathbb{R}^{2d}$.

From now on,  $ E= \{\rho u: \rho \in \mathcal{M} \}$, where $u(t)= e^{-|t|^2}, e^{-|t|}, 1/(1+4\pi^2t^2)$, or $1/\cosh(\pi t)$.  Let $v \in E$. It is known that  $\mathcal{G}(v, a_1, b_1)$ is a frame in $L^2(\mathbb{R})$  if and only if  $  a_1b_1<1$ \cite{Jan1,Jan2,Lyu,Sei}. Using this result and Proposition \ref{Gabor2}, we can prove the following corollary, which proves Example \ref{examp} in the introduction.

\begin{cor}\label{density1} Let $d>1$. For each $j \in \{1,\dots,d \}$, let $g_{j,1}, g_{j,2}\in L^2(\mathbb{R})$ be two independent functions. Assume that there are  $i, k \in \{1, \dots, d \}$ such that $i \neq k$ and  $g_{i,1},g_{k,2} \in E$.  If  $\mathcal{G}(g, a_1, b_1 )$ is a frame, then $ a_1b_1<1$, where
\begin{eqnarray*}
g= \bigotimes_{j=1}^{d} g_{j,1}+\bigotimes_{j=1}^{d} g_{j,2}.
\end{eqnarray*}
\end{cor}

Corollary \ref{2sequences} (i) also implies the following corollary, in which the rank of the window function is not necessarily finite.

\begin{cor}\label{density2} Let $d>2$. For each $j \in \{1,2 \}$, let $g_{j,1}, g_{j,2}\in L^2(\mathbb{R})$ be two independent functions  such that $g_{1,1}, g_{2,2} \in E$ or $g_{1,2}, g_{2,1} \in E$. Let $g_{1}, g_{2}\in L^2(\mathbb{R}^{d-2})$ be two independent functions.  If  $\mathcal{G}(g, a_1, b_1 )$ is a frame, then $ a_1b_1<1$, where $$g=  g_{1,1}\otimes g_{2,1}\otimes g_1  +g_{1,2}\otimes g_{2,2}\otimes g_2.$$
\end{cor}

\begin{examp} Let $g_{1}, g_{2}\in L^2(\mathbb{R}^{d-2})$ be two independent functions. let $u,v \in L^2(\mathbb{R})$ such that $u$ and $1/\cosh(\pi \cdot)$
are independent and $v$ and $1/(1+4\pi^2 (\cdot)^2)$ are independent. The following function satisfies the hypotheses of Corollary \ref{density2}.
\begin{eqnarray*}
g(x_1,\dots , x_d) = \frac{u(x_2) g_1(x_3,\dots , x_d ) }{1+4\pi^2 x_1^2}+ \frac{v(x_1) g_2(x_3,\dots , x_d)}{\cosh(\pi x_2)}.
\end{eqnarray*}
\end{examp}

\end{document}